\documentclass{amsart}
\usepackage{amssymb,amsmath,latexsym,mathrsfs}
\usepackage{times}

\numberwithin{equation}{section}

\newtheorem{thm}{Theorem}[section]

\newtheorem{cor}[thm]{Corollary}
\newtheorem{conj}[thm]{Conjecture}

\newtheorem{lemma}[thm]{Lemma}

\newcommand{\ssp}{\hspace{1pt}}
\newcommand{\del}{\backslash}
\newcommand{\cl}{\hbox{\rm cl}}
\newcommand{\ch}{\hbox{\rm ch}}
\newcommand{\ba}{\hbox{\rm b}}
\newcommand{\mbA}{\mathbf{A}}
\newcommand{\mbB}{\mathbf{B}}
\newcommand{\mcA}{\mathcal{A}}
\newcommand{\mcB}{\mathcal{B}}

\newcommand{\mcZ}{\mathcal{Z}}

\title{Sparse Paving Matroids, Basis-Exchange Properties, and Cyclic
  Flats}
\author[J.~Bonin]{Joseph E.~Bonin}
\address{Department of Mathematics\\ The George Washington University\\
  Washington, D.C. 20052}
\email{jbonin@gwu.edu}
\date{\today}

\begin{document}

\begin{abstract}
  We provide evidence for five long-standing, basis-exchange
  conjectures for matroids by proving them for the enormous class of
  sparse paving matroids.  We also explore the role that these
  matroids may play in the following problem: as a function of the
  size of the ground set, what is the greatest number of cyclic flats
  that a matroid can have?
\end{abstract}

\maketitle

\markboth{J.~Bonin, \emph{Sparse Paving Matroids, Basis-Exchange
    Properties, and Cyclic Flats}} {J.~Bonin, \emph{Sparse Paving
    Matroids, Basis-Exchange Properties, and Cyclic Flats}}

\section{Introduction}\label{sec:intro}

A matroid is \emph{paving} if the closure of each nonspanning circuit
is a hyperplane; it is \emph{sparse paving} if each nonspanning
circuit is a hyperplane.  Thus, a matroid $M$ of rank $r$ is sparse
paving if and only if each $r$-subset of $E(M)$ is either a basis or a
circuit-hyperplane.  It follows that the class of sparse paving
matroids is dual-closed.  It is easy to show that this class is also
minor-closed.  Sparse paving matroids can also be characterized as the
matroids $M$ for which both $M$ and its dual, $M^*$, are paving.

While paving and sparse paving matroids have received increasing
attention recently (see, e.g., \cite{rota,jerrum,asy,catalog,hvec}),
they have long played important roles in matroid theory.  For
instance, D.~Knuth~\cite{knuth} constructed at least
\begin{equation*}
  \frac{2^{\binom{n}{\lfloor n/2 \rfloor}/2n}}{n!}
\end{equation*}
nonisomorphic sparse paving matroids of rank $\lfloor n/2 \rfloor$ on
$n$ elements; with the upper bound by M.~Piff \cite{piff}, it follows
that the number $g_n$ of nonisomorphic simple matroids on $n$ elements
satisfies
\begin{equation}\label{eq:bound}
  n -\frac{3}{2}\log_2 n + O(\log_2\log_2 n) \leq \log_2\log_2 g_n \leq n
  -\log_2 n + O(\log_2\log_2 n),
\end{equation}
with sparse paving matroids accounting for the lower bound.  Taking
this further, in~\cite{asy}, D.~Mayhew, M.~Newman, D.~Welsh, and
G.~Whittle have conjectured that, asymptotically, almost all matroids
are sparse paving.

The five basis-exchange conjectures treated in this paper, all of
which have been open for decades and have been proven for only a few
classes of matroids, are part of the circle of ideas that revolve
around the well-known symmetric basis-exchange property: for any bases
$B_1$, $B_2$ of a matroid $M$, if $b_1\in B_1-B_2$, then, for some
$b_2\in B_2-B_1$, both $(B_1-b_1)\cup b_2$ and $(B_2-b_2)\cup b_1$ are
also bases of $M$.

The first conjecture concerns the \emph{basis pair graph}, $G(M)$, of
a matroid $M$, which is defined as follows.  The vertices of $G(M)$
are the ordered triples $(A_1,A_2,A_3)$ of subsets of $E(M)$ where
$A_1$ and $A_2$ are disjoint bases of $M$ and $A_3$ is $E(M)-(A_1\cup
A_2)$.  (Thus, the inequality $|E(M)|\geq 2\,r(M)$ must hold in order
for $G(M)$ to have any vertices.)  Two vertices, say $\mbA =
(A_1,A_2,A_3)$ and $\mbB = (B_1,B_2,B_3)$, of $G(M)$ are adjacent if
$\mbB$ can be obtained from $\mbA$ by switching some pair of elements
in two different sets in $\mbA$, that is, if
\begin{equation*}
  |A_1-B_1|+|A_2-B_2|+|A_3-B_3|=2.
\end{equation*}
If $E(M)$ is the disjoint union of two bases of $M$, then $G(M)$ is
isomorphic to the basis-cobasis graph studied by R.~Cordovil and
M.~Moreira~\cite{bc}.  The following conjecture was posed by
M.~Farber~\cite{faber}, who proved it for transversal matroids.
(In~\cite{frs}, M.~Farber, B.~Richter and H.~Shank proved it for
graphic and cographic matroids.)

\begin{conj}\label{conj:faber}
  The basis pair graph of any matroid is connected.
\end{conj}

The second conjecture involves a family of graphs that we can
associate with a matroid.  Fix an integer $k\geq 2$.  Let $M$ be a
matroid of rank $r$ and let $S$ be a multiset of size $kr$ with
elements in $E(M)$.  Define the graph $G_M(S)$ as follows: the
vertices of $G_M(S)$ are all multisets of $k$ bases of $M$ whose
multiset union is $S$; two vertices are adjacent if one can be
obtained from the other by one symmetric exchange among one pair of
bases in one of the vertices.  Thus, vertices
$\mcA=\{A_1,A_2,\ldots,A_k\}$ and $\mcB=\{B_1,B_2,\ldots,B_k\}$ are
adjacent if, for some bases $B_i,B_j\in\mcB$ and elements $b_i\in
B_i-B_j$ and $b_j\in B_j-B_i$, we obtain $\mcA$ from $\mcB$ by
replacing $B_i$ by $(B_i-b_i)\cup b_j$ and replacing $B_j$ by
$(B_j-b_j)\cup b_i$.  (This graph may be empty.)  The conjecture below
is due to N.~White~\cite[Conjecture 12]{white}.

\begin{conj}\label{conj:white}
  For any matroid $M$ and multiset $S$ of size $k\,r(M)$ with elements
  in $E(M)$ and with $k\geq 2$, the graph $G_M(S)$ is connected.
\end{conj}

Conjecture~\ref{conj:white} is sometimes cast in terms of toric
ideals.  A routine argument shows that the conjecture holds for $M$ if
and only if it holds for $M^*$.  It has been shown for graphic (and so
for cographic) matroids by J.~Blasiak~\cite{blasiak} and for rank-$3$
(and so for nullity-$3$) matroids by K.~Kashiwabara~\cite{kk}.
J.~Herzog and T.~Hibi~\cite{hh} have shown that
Conjecture~\ref{conj:white} is equivalent to its counterpart for
discrete polymatroids.  J.~Schweig~\cite{jay} has proven the
counterpart of the conjecture for certain discrete polymatroids.

While Conjecture~\ref{conj:white} has received most attention,
\cite[Conjecture 12]{white} has three parts, of which the next
conjecture is the strongest.  Consider the graph $G'_M(S)$ in which
$k$-tuples of bases replace multisets of bases.  Thus, its vertices
are all $k$-tuples of bases of $M$ whose multiset union is $S$;
vertices $\mbA=(A_1,A_2,\ldots,A_k)$ and $\mbB=(B_1,B_2,\ldots,B_k)$
are adjacent if, for some integers $i$ and $j$ with $1\leq i<j\leq k$
and some $b_i\in B_i-B_j$ and $b_j\in B_j-B_i$, we obtain $\mbA$ from
$\mbB$ by replacing $B_i$ by $(B_i-b_i)\cup b_j$ and replacing $B_j$
by $(B_j-b_j)\cup b_i$.

\begin{conj}\label{conj:white2}
  For any matroid $M$ and multiset $S$ of size $k\,r(M)$ with elements
  in $E(M)$ and with $k\geq 2$, the graph $G'_M(S)$ is connected.
\end{conj}

We show that the conclusion of Conjecture~\ref{conj:white2} holds for
a matroid $M$ if Conjecture~\ref{conj:white} holds for $M$ and
Conjecture~\ref{conj:faber} holds for all of its minors.  It follows
that Conjecture~\ref{conj:white2} holds for all sparse paving
matroids.

The fourth conjecture was made by Y.~Kajitani, S.~Ueno, and
H.~Miyano~\cite{orderingKUM}.  A matroid $M$ is \emph{cyclically
  orderable} if there is a cyclic permutation $(a_1,a_2,\ldots,a_n)$
of $E(M)$ in which each set of $r(M)$ cyclically-consecutive elements
is a basis of $M$.

\begin{conj}\label{conj:co}
  A matroid $M$ is cyclically orderable if and only if, for all
  nonempty subsets $A$ of $E(M)$, 
  \begin{equation}\label{eq:unifd}
    r(M)\, |A| \leq r(A)\,|E(M)|.
  \end{equation}
\end{conj}

A counting argument shows that inequality (\ref{eq:unifd}) holds if
$M$ is cyclically orderable.  Recent progress on this conjecture was
made by J.~van den Heuvel and S.~Thomass\'e~\cite{cyclicHT}

The fifth conjecture was first raised as a problem by
H.~Gabow~\cite{gabow} and has been pursued
in~\cite{bc,cyclicHT,orderingKUM}.  To match our work below, we state
the conjecture in the case of disjoint bases; it is easy to show that
this implies its counterpart for arbitrary bases.

\begin{conj}\label{conj:cyclic}
  If $B_1$ and $B_2$ are disjoint bases of a rank-$r$ matroid $M$,
  then some cycle $(b_1,b_2,\ldots,b_r,b_{r+1},\ldots,b_{2r})$ has
  $B_1=\{b_1,b_2,\ldots,b_r\}$ and $B_2 =
  \{b_{r+1},b_{r+2},\ldots,b_{2r}\}$, and has each set of $r$
  cyclically-consecutive elements being a basis of $M$.
\end{conj}

It is not hard to show that if this conjecture holds for $M$, then it
holds for $M^*$ and for all minors of $M$.  H.~Gabow~\cite{gabow}
noted that the conjecture holds for transversal matroids.  It has also
been proven for graphic matroids~\cite{bc,orderingKUM}.
A.~de~Mier~\cite{anna} observed that this conjecture holds for
strongly base-orderable matroids.  Recall that a matroid is
\emph{strongly base-orderable} if for each pair of bases $B_1$ and
$B_2$ of $M$, there is a bijection $\phi:B_1\rightarrow B_2$ such that
for every subset $X\subseteq B_1$, both $(B_1-X)\cup \phi(X)$ and
$(B_2-\phi(X))\cup X$ are bases.  If $M$ is strongly base-orderable,
then listing the elements of $B_1$ in any order followed by their
images under $\phi$, in the corresponding order, gives the required
cycle.  The class of strongly base-orderable matroids is both
minor-closed and dual-closed, and it strictly contains the class of
all gammoids (which include transversal matroids).

In Section~\ref{sec:proofs}, we prove
Conjectures~\ref{conj:faber}--\ref{conj:cyclic} for sparse paving
matroids.  Section~\ref{sec:ncf} treats another aspect of these
matroids as we study the greatest number of cyclic flats in any
matroid on $n$ elements.  We give an upper bound on this number and
note that a lower bound follows from work of R.~Graham and
N.~Sloane~\cite{gs} which, in a different setting, essentially
constructs sparse paving matroids.  The gap between these bounds is
similar to that in inequality (\ref{eq:bound}).  We provide the
relevant background on cyclic flats in that section.

Our notation follows J.~Oxley~\cite{ox}.  The symmetric difference,
$(X-Y)\cup (Y-X)$, of two sets $X$ and $Y$ is denoted by $X\triangle
Y$.  We let $[n]$ denote the set $\{1,2,\ldots,n\}$.

\section{Proofs of Conjectures~\ref{conj:faber}--\ref{conj:cyclic} in
  the Case of Sparse Paving Matroids}\label{sec:proofs}

We will use the lemmas below.  The first follows easily from the
definition of sparse paving.

\begin{lemma}\label{lem:chb}
  Let $M$ be a sparse paving matroid of rank $r$.  Let $H$ and $B$ be
  two $r$-subsets of $E(M)$ with $|H\triangle B|=2$.  If $H$ is a
  circuit-hyperplane of $M$, then $B$ is a basis.
\end{lemma}

Although we will not use it, we note that the following strengthening
of Lemma~\ref{lem:chb} is easy to prove: a matroid $M$ of rank $r$ is
sparse paving if and only if whenever $H$ and $B$ are two $r$-subsets
of $E(M)$ with $|H\triangle B|=2$ and $H$ is not a basis, then $B$ is
a basis.  (We remark that the analogous condition on discrete
polymatroids winds up being too restrictive to be of interest.)

\begin{lemma}\label{lem:swap}
  Let $B$ and $B'$ be distinct bases of a sparse paving matroid $M$.
  For $a\in B-B'$ and $X\subseteq B'-B$, there are at least $|X|-2$
  elements $x\in X$ for which both $(B-a)\cup x$ and $(B' - x)\cup a$
  are bases of $M$.
\end{lemma}

\begin{proof}
  The lemma follows since, by Lemma~\ref{lem:chb}, at most one set
  $(B-a)\cup x$ with $x\in X$, and at most one set $(B' - x')\cup a$
  with $x'\in X$, is a circuit-hyperplane.
\end{proof}

We now turn to Conjecture ~\ref{conj:faber}.

\begin{thm}\label{thm:faber}
  Conjecture ~\ref{conj:faber} holds for sparse paving matroids.
\end{thm}

\begin{proof}
  We first prove the result when $E(M)$ is the disjoint union of two
  bases; we will then reduce the general case to this one.  In this
  case, vertices have the form $(B_1,B_2,\emptyset)$, which we
  simplify to $(B_1,B_2)$ in the next two paragraphs.  We must show
  that for each pair $(A_1,A_2)$ and $(B_1,B_2)$ of vertices in $G(M)$
  with $|A_1\triangle B_1|\geq 4$, there is a path between them.  For
  this, it suffices to show that there is a path from $(B_1,B_2)$ to a
  vertex $(B'_1,B'_2)$ with $|A_1\triangle B'_1|<|A_1\triangle B_1|$.

  If $|B_1 - A_1|\geq 3$, then fix $x\in B_1-A_1$ and set $X=A_1-B_1$.
  We have $|X|\geq 3$ and $X\subseteq B_2$, so, by Lemma
  \ref{lem:swap}, the pair $\bigl((B_1-x)\cup y,(B_2-y)\cup x\bigr)$
  is a vertex of $G(M)$ for some $y\in X$.  Also, $|A_1\triangle
  \bigl((B_1-x)\cup y\bigr)|<|A_1\triangle B_1|$, as needed.

  In the remaining case, $|B_1 - A_1|=2$, let $B_1-A_1 = \{b_1,b_2\}$
  and $A_1-B_1 = \{a_1,a_2\}$.  Thus, $a_1,a_2\in B_2$.  If any of the
  following four symmetric exchanges yields only bases, it would
  provide the desired vertex $(B'_1,B'_2)$ adjacent to $(B_1,B_2)$:
  \begin{enumerate}
  \item[(a)] $(B_1-b_1)\cup a_1$ and $(B_2 - a_1)\cup b_1$,
  \item[(b)] $(B_1-b_1)\cup a_2$ and $(B_2 - a_2)\cup b_1$,
  \item[(c)] $(B_1-b_2)\cup a_1$ and $(B_2 - a_1)\cup b_2$,
  \item[(d)] $(B_1-b_2)\cup a_2$ and $(B_2 - a_2)\cup b_2$.
  \end{enumerate}
  Thus, we may assume that each pair contains a circuit-hyperplane.
  By symmetry, we may assume that $(B_1-b_1)\cup a_1$ is a
  circuit-hyperplane; then $(B_1-b_1)\cup a_2$ and $(B_1-b_2)\cup a_1$
  are bases by Lemma~\ref{lem:chb}, so $(B_2 - a_2)\cup b_1$ and $(B_2
  - a_1)\cup b_2$ are circuit-hyperplanes; thus, $(B_2 - a_2)\cup b_2$
  is a basis by Lemma~\ref{lem:chb}, so $(B_1-b_2)\cup a_2$ is a
  circuit-hyperplane.  For all four sets just identified to be
  circuit-hyperplanes, we must have $r(M)\geq 3$, so there is an
  element $x$ in $B_1\cap A_1$.  By comparison with the four known
  circuit-hyperplanes, it follows that each set in the following
  symmetric exchanges is a basis:
  \begin{enumerate}
  \item[(e)] $(B_1-x)\cup a_1$ and $(B_2 - a_1)\cup x$,
  \item[(f)] $B'_1=(B_1-\{x,b_2\})\cup \{a_1,a_2\}$ and $B'_2=(B_2 -
    \{a_1,a_2\})\cup \{x,b_2\}$.
  \end{enumerate}
  Since $(B'_1,B'_2)$ is adjacent to $(A_1,A_2)$, the needed path from
  $(B_1,B_2)$ to $(A_1,A_2)$ exists.

  In the general case, for two vertices $\mbA = (A_1,A_2,A_3)$ and
  $(B_1,B_2,B_3)$ of $G(M)$, we will show that there is a path in
  $G(M)$ from $\mbA$ to a vertex of the form $(C_1,C_2,B_3)$; the
  theorem then follows by applying the case just treated to the basis
  pair graph of $M\del B_3$.  (Recall that the third set in these
  triples need not be a basis.)

  Assume $|A_3\triangle B_3|\geq 4$.  By symmetry, we may assume
  $|A_1\cap B_3|\geq 1$; fix some $a_1\in A_1\cap B_3$.  Since $M$ is
  sparse paving, the hyperplane $\cl(A_1-a_1)$ contains at most one
  element in $A_3-B_3$, so $A'_1=(A_1-a_1)\cup a_3$ is a basis for
  some $a_3\in A_3-B_3$.  The vertex $(A'_1,A_2,A'_3)$, where $A'_3 =
  (A_3-a_3)\cup a_1$, is adjacent to $\mbA$ and has $|A'_3\triangle
  B_3| <|A_3\triangle B_3|$.

  By iterating the argument above, it now suffices to treat the case
  $|A_3\triangle B_3|=2$.  Let $A_3 - B_3 = \{a_3\}$ and $B_3 - A_3 =
  \{b_3\}$.  We may assume $b_3\in A_1$.  If $(A_1-b_3)\cup a_3$ is a
  basis of $M$, then the claim holds, so assume instead that this set
  is a circuit-hyperplane.  By symmetrically exchanging any element
  $a_1\in A_1-b_3$ with some element $a_2\in A_2$, we get a vertex
  $((A_1 - a_1)\cup a_2,(A_2 - a_2)\cup a_1,A_3)$ that is adjacent to
  $\mbA$ and in which, by Lemma~\ref{lem:chb}, we can exchange $b_3$
  in $(A_1 - a_1)\cup a_2$ with $a_3$ in $A_3$, which completes the
  proof of the claim and so of the theorem.
\end{proof}

We now turn to Conjecture~\ref{conj:white}.

\begin{thm}\label{thm:white}
  Conjecture ~\ref{conj:white} holds for sparse paving matroids.
\end{thm}

\begin{proof}
  Let $M$ be a sparse paving matroid.  We prove that $G_M(S)$ is
  connected by induction on $k$, where $|S|=k\, r(M)$.  The base case
  $k=1$ is trivial: $G_M(S)$ is connected since it has at most one
  vertex.  For $k\geq 2$, we claim that for any two vertices
  \begin{equation*}
    \mcA=\{A_1,A_2,\ldots,A_k\} \quad \text{ and }  \quad
    \mcB=\{B_1,B_2,\ldots,B_k\}
  \end{equation*}
  of $G_M(S)$, there are (possibly trivial) paths from $\mcA$ to some
  vertex $\{A'_1,A'_2,\ldots,A'_k\}$ and from $\mcB$ to some vertex
  $\{B'_1,B'_2,\ldots,B'_k\}$ with $A'_1=B'_1$.  Proving this claim
  gives the result by induction since having a path from $\mcA$ to
  $\mcB$ in $G_M(S)$ follows from having a path from
  $\{A'_2,A'_3,\ldots,A'_k\}$ to $\{B'_2,B'_3,\ldots,B'_k\}$ in
  $G_M(S-A'_1)$, where $S-A'_1$ is the multiset difference.  List the
  sets in $\mcA$ and $\mcB$ so that $|A_1\triangle B_1|\leq
  |A_h\triangle B_j|$ for all $h,j\in[k]$.  Set $|A_1\triangle
  B_1|=2i$.  To prove the claim, it suffices to show that if $i>0$,
  then
  \begin{itemize}
  \item[(*)] there is a path from $\mcB$ to a vertex
    $\{B''_1,B''_2,\ldots,B''_k\}$ with $|A_1\triangle B''_1|<2i$.
  \end{itemize}

  Set $A_1-B_1=\{a_1,a_2,\ldots,a_i\}$ and
  $B_1-A_1=\{b_1,b_2,\ldots,b_i\}$.  By symmetry, we may assume that
  the sum of the multiplicities of the elements in $A_1-B_1$ in $S$ is
  at least as large as the corresponding sum for $B_1-A_1$.  It
  follows that some basis in $\mcB$, say $B_2$, has more elements from
  $A_1-B_1$ than from $B_1-A_1$.  We consider several options for
  $B_2$.

  For the case $i\geq 3$, first assume $B_2\cap( B_1-A_1) =\emptyset$.
  We may assume $a_1\in B_2$.  Apply Lemma \ref{lem:swap} with $x =
  a_1$ and $X = B_1-A_1$ (so $|X|\geq 3$): for some $b_h\in B_1-A_1$,
  both $(B_1-b_h)\cup a_1$ and $(B_2-a_1)\cup b_h$ are bases, so
  statement (*) follows.

  Now, along with $i\geq 3$, assume $|B_2\cap (A_1-B_1)|\geq 3$.  Let
  $X = B_2\cap (A_1-B_1)$.  Since $B_2$ has more elements from
  $A_1-B_1$ than from $B_1-A_1$, some element in $B_1-A_1$, say $b_1$,
  is not in $B_2$.  Apply Lemma \ref{lem:swap} to $B_1$ and $B_2$ with
  $x=b_1$ and $X$: for some $a_h\in X$, both $(B_1-b_1)\cup a_h$ and
  $(B_2-a_h)\cup b_1$ are bases. Statement (*) now follows.

  We now address the case with $B_2\cap(A_1\triangle B_1)=
  \{a_1,a_2,b_3\}$, thereby completing the argument for $i\geq 3$.  If
  we can symmetrically exchange one of $a_1,a_2$ in $B_2$ for one of
  $b_1,b_2$ in $B_1$ to get bases, then statement (*) holds.  Assume
  that none of these four symmetric exchanges yields only bases.  An
  argument like that in the third paragraph of the proof of Theorem
  \ref{thm:faber} shows that we may assume that
  \begin{equation*}
    (B_1-b_1)\cup a_1, \quad (B_2 - a_2)\cup b_1, \quad 
    (B_2 - a_1)\cup b_2, \quad \text{and} \quad (B_1-b_2)\cup a_2
  \end{equation*}
  are circuit-hyperplanes.  In order to have $|A_1\triangle B_1|\leq
  |A_1\triangle B_2|$ given that $B_2\cap(A_1\triangle B_1)$ is
  $\{a_1,a_2,b_3\}$, there must be an element, say $y$, in $B_2-
  (A_1\cup B_1)$.  From Lemma~\ref{lem:chb} and the
  circuit-hyperplanes above, we have that $(B_1-b_1)\cup y$ and $(B_2
  - y)\cup b_1$ are bases, as are $(B_1-\{b_1,b_2\})\cup \{y,a_1\}$
  and $(B_2 - \{y,a_1\})\cup \{b_1,b_2\}$.  Statement (*) now follows,
  which completes the argument for $i\geq 3$.

  Now assume $i=2$.  By symmetry, there are two cases: $B_2 \cap
  \{b_1,b_2\}$ is either $\emptyset$ or $\{b_1\}$.  First assume
  $B_2\cap \{b_1,b_2\} = \emptyset$.  We may assume $a_1\in B_2$.  If
  $a_1$ in $B_2$ can be symmetrically exchanged with either $b_1$ or
  $b_2$ in $B_1$ to yield two bases, then statement (*) holds, so
  assume this fails.  By symmetry, $H_1=(B_1-b_1)\cup a_1$ and
  $H_2=(B_2 - a_1)\cup b_2$ can be assumed to be circuit-hyperplanes.
  Since $|A_1\triangle B_1|\leq |A_1\triangle B_2|$, there are least
  two elements, say $z_2$ and $z_3$, in $B_2-A_1$.  By
  Lemma~\ref{lem:chb}, either $(B_2 - z_2)\cup b_1$ or $(B_2 -
  z_3)\cup b_1$ is a basis; assume the former is.  Comparison with
  $H_1$ shows that $(B_1-b_1)\cup z_2 $ and $(B_1-\{b_1,b_2\})\cup
  \{z_2,a_1\}$ are bases; similarly, $(B_2 - \{z_2,a_1\})\cup
  \{b_1,b_2\}$ is a basis by comparison with $H_2$.  Statement (*) now
  follows.

  We now address the case with $B_2\cap \{b_1,b_2\}=\{b_1\}$, thus
  completing the argument for $i=2$.  Note that $B_2$ must also
  contain $a_1$ and $a_2$.  Statement (*) holds if $b_2$ in $B_1$ can
  be symmetrically exchanged with either $a_1$ or $a_2$ in $B_2$ to
  yield two bases.  If neither exchange yields only bases, then, by
  symmetry, we may assume that $H_1=(B_1-b_2)\cup a_1$ and $H_2=(B_2 -
  a_2)\cup b_2$ are circuit-hyperplanes.  At least two elements in
  $A_1\cap B_1$, say $x_3$ and $x_4$, are not in $B_2$ since
  $|A_1\triangle B_1|\leq |A_1\triangle B_2|$.  At least one of $(B_2
  - a_1)\cup x_3$ and $(B_2 - a_1)\cup x_4$ is a basis by
  Lemma~\ref{lem:chb}; assume the first is.  Now $(B_1-x_3)\cup a_1$
  is a basis by comparison with $H_1$.  The sets
  $(B_1-\{x_3,b_2\})\cup \{a_1,a_2\}$ and $(B_2 - \{a_1,a_2\})\cup
  \{x_3,b_2\}$ are also bases by comparison with $H_1$ and $H_2$,
  respectively.  It follows that statement (*) holds.  This completes
  the argument for $i=2$.

  Finally, assume $i=1$, so $A_1-B_1=\{a_1\}$ and $B_1-A_1=\{b_1\}$.
  Thus, $B_2$ contains $a_1$ and not $b_1$.  Let $X = B_2-a_1$.  If
  $X\cup b_1$ is a basis (as it must be if $k$ is $2$), then
  exchanging $a_1$ and $b_1$ in $B_2$ and $B_1$ shows that statement
  (*) holds.  Thus, assume $k\geq 3$ and
  \begin{itemize}
  \item[(A)] $X\cup b_1$ is a circuit-hyperplane.
  \end{itemize}
  If $3\leq h\leq k$ and $b_1\not\in B_h$, and if there is an element
  $y\in X-B_h$, then there is a $z\in B_h-B_2$ for which both $(B_h
  -z)\cup y$ and $(B_2 -y)\cup z$ are bases; from Lemma~\ref{lem:chb}
  and statement (A), it follows that we can symmetrically exchange
  $a_1$ in $(B_2 -y)\cup z$ with $b_1$ in $B_1$ to get two bases,
  which yields statement (*).  Thus, we may assume
  \begin{itemize}
  \item[(B)] each basis $B_h$ contains either $b_1$ or all of $X$.
  \end{itemize}
  If $B_h\cap\{a_1,b_1\}=\{b_1\}$ for some $h$ with $3\leq h\leq k$,
  then the assumption about the multiplicities of $a_1$ and $b_1$
  implies that $B_{h'}\cap\{a_1,b_1\}=\{a_1\}$ for some $h'$ with
  $3\leq h'\leq k$.  Symmetrically exchange $a_1$ in $B_{h'}-B_h$ for
  some $z\in B_h-B_{h'}$ to get bases; since $B_{h'}-a_1$ is $X$ by
  statement (B), statement (A) gives $z\ne b_1$.  Thus, we may assume
  \begin{itemize}
  \item[(C)] for $3\leq h\leq k$, if $b_1\in B_h$, then $a_1\in B_h$.
  \end{itemize}
  Assume $3\leq h\leq k$ and $a_1,b_1\in B_h$.  If $|B_2\triangle
  B_h|\geq 4$, then for $x\in (B_h-b_1)-B_2$, we can symmetrically
  exchange $x\in B_h$ with some $y\in B_2$ (which cannot be $a_1$) to
  yield two bases; with statement (A), this allows us to exchange
  $b_1$ in $B_1$ with $a_1$ in $(B_2-y)\cup x$ to yield statement (*).
  Thus, we may assume
  \begin{itemize}
  \item[(D)] if $a_1, b_1 \in B_h$, then $|B_2\triangle B_h|=2$.
  \end{itemize}
  The proof is completed by showing that statements (A)--(D) yield a
  contradiction.  Consider the multisets $\mathscr{A} =
  \{\{a_1\},A_2,A_3,\ldots,A_k\}$ and
  $\mathscr{B}=\{\{b_1\},B_2,B_3,\ldots,B_k\}$ of sets.  Their
  multiset unions, $\bigcup_{A\in \mathscr{A}}A$ and $\bigcup_{B\in
    \mathscr{B}}B$, are equal.  Let $b_1$ have multiplicity $t+1$ in
  these unions.  Statements (B)--(D) imply that the sum of the
  multiplicities of the elements in $X$ in the sets in $\mathscr{B}$
  is $|X|(k-t-1)+(|X|-1)t$, that is, $|X|(k-1)-t$.  By statement (A),
  $X\cup b_1$ is not in $\mathscr{A}$, so the sum of the
  multiplicities of the elements in $X$ in the sets in $\mathscr{A}$
  is at most $|X|(k-t-2)+(|X|-1)(t+1)$, that is, $|X|(k-1)-t-1$,
  which, as desired, contradicts the equality $\bigcup_{A\in
    \mathscr{A}}A=\bigcup_{B\in \mathscr{B}}B$.
\end{proof}

We now prove a general connection between Conjectures
\ref{conj:faber}, \ref{conj:white}, and \ref{conj:white2}.

\begin{thm}
  Let $M$ be a matroid for which the basis pair graph of each of its
  minors is connected.  For $k\geq 2$, let $S$ be a multiset of size
  $k\,r(M)$ with elements in $E(M)$.  If $G_M(S)$ is connected, then
  so is $G'_M(S)$.
\end{thm}

\begin{proof}
  Since $G_M(S)$ is connected, to show that $G'_M(S)$ is connected it
  suffices to show that for each vertex $\mbA =(A_1,A_2,\ldots,A_k)$
  of $G'_M(S)$ and each permutation $\sigma$ of $[k]$, there is a path
  in $G'_M(S)$ from $\mbA$ to
  $\mbA_{\sigma}=(A_{\sigma(1)},A_{\sigma(2)},\ldots,A_{\sigma(k)})$.
  Since every permutation is a composition of transpositions, we focus
  on a transposition $\sigma$, say permuting $i$ and $j$ with $i<j$.
  The desired result follows if we show that there is a path from
  $\mbA$ to $\mbA_{\sigma}$ in which all bases but the $i$-th and
  $j$-th are fixed.  This follow by noting that the sequence of
  symmetric exchanges that gives a path from
  $(A_i-A_j,A_j-A_i,\emptyset)$ to $(A_j-A_i,A_i-A_j,\emptyset)$ in
  the basis pair graph of the minor $M|(A_i\cup A_j)/(A_i\cap A_j)$
  also gives the desired path from $\mbA$ to $\mbA_{\sigma}$ in
  $G'_M(S)$.
\end{proof}

\begin{cor}
  For any minor-closed class of matroids for which
  Conjecture~\ref{conj:faber} holds, Conjectures \ref{conj:white} and
  \ref{conj:white2} are equivalent.  In particular,
  Conjecture~\ref{conj:white2} holds for all sparse paving matroids.
\end{cor}

For Conjecture~\ref{conj:co}, we start with a lemma.  A
$k$-\emph{interval} in a cycle $\sigma$ is a set of $k$
cyclically-consecutive elements, that is,
$\{x,\sigma(x),\sigma^2(x),\ldots,\sigma^{k-1}(x)\}$ for some $x$.

\begin{lemma}\label{lem:av}
  Let $M$ be a rank-$r$ sparse paving matroid on $n$ elements.  If
  $2r\leq n$, then, over all cycles on $E(M)$, the average number of
  $r$-intervals that are circuit-hyperplanes of $M$ is less than two.
\end{lemma}

\begin{proof}
  Let $\ba(M)$ and $\ch(M)$ be, respectively, the numbers of bases and
  circuit-hyperplanes of $M$.  By focusing on circuit-hyperplanes, it
  follows that the average of interest is
  \begin{equation*}
    \frac{\ch(M)\, r!\, (n-r)!}{(n-1)!}.
  \end{equation*}
  The desired result follows easily from this expression and the
  assumed inequality, $2r\leq n$, once we show
  \begin{equation}\label{in:ch}
    \ch(M)\leq \frac{1}{n-r+1}\binom{n}{r}.
  \end{equation}
  Consider the pairs $(H,B)$ consisting of a circuit-hyperplane $H$ of
  $M$ and a basis $B$ of $M$ with $|H\triangle B|=2$.  The definition
  of sparse paving gives three properties that yield the inequality
  above: each circuit-hyperplane is in $r(n-r)$ such pairs, each basis
  is in at most $r$ such pairs, and $\ba(M)+\ch(M)=\binom{n}{r}$.
\end{proof}

\begin{thm}\label{thm:co}
  Conjecture~\ref{conj:co} holds for sparse paving matroids.
\end{thm}

\begin{proof}
  As noted after Conjecture~\ref{conj:co}, inequality~(\ref{eq:unifd})
  holds in every cyclically orderable matroid.  The conjecture is easy
  to verify for all sparse paving matroids that have rank or nullity
  at most two (this includes all disconnected sparse paving matroids,
  i.e., $U_{0,n}$, $U_{n,n}$, $U_{n-1,n}\oplus U_{1,1}$,
  $U_{1,n}\oplus U_{0,1}$, and $U_{1,2}\oplus U_{1,2}$; this also
  includes all cases in which inequality~(\ref{eq:unifd}) fails), so
  below we assume that $M$ has rank and nullity at least three.

  We may assume $E(M) = [n]$.  For a cycle $\sigma$ on $E(M)$, all
  $r(M)$-intervals in $\sigma$ are bases of $M$ if and only if their
  complements, all $r(M^*)$-intervals in $\sigma$, are bases of $M^*$,
  so, by replacing $M$ by $M^*$ if needed, we may assume that $2r\leq
  n$ where $r=r(M)$.  By Lemma~\ref{lem:av}, for some cycle, say
  $\sigma_1=(1,2,\ldots,n)$, on $E(M)$, at most one of its
  $r$-intervals is a circuit-hyperplane.  We may assume there is such
  an interval, say
  \begin{equation*}
    H_1=\{4,5,\ldots,r+3\},
  \end{equation*}
  otherwise the desired conclusion holds.  

  Consider $\sigma_2=(1,2,\underline{4},\underline{3},5,\ldots,n)$.
  (To aid the reader, we underline the entries that differ from
  $\sigma_1$.)  Only two of its $r$-intervals differ from their
  counterparts in $\sigma_1$, namely, $\{3,5,6,\ldots,r+3\}$, which is
  a basis (use Lemma~\ref{lem:chb} with $H_1$), and
  \begin{equation*}
    H_2=\{n-r+4,\ldots,n,1,2,4\}.
  \end{equation*}
  If $H_2$ is a basis, then $\sigma_2$ is the cycle we want.  Thus,
  assume that $H_2$ is a circuit-hyperplane.

  We repeatedly apply this type of argument below.  For brevity, for
  each cycle we list its $r$-intervals that differ from their
  counterparts in $\sigma_1$ and, when possible, the
  circuit-hyperplanes that, with Lemma~\ref{lem:chb}, show that these
  intervals are bases.  For brevity, we omit the $r$-interval
  $\{i,5,6,\ldots,r+3\}$, with $i\ne 4$, which is a basis (compare it
  to $H_1$).  Since the permutations $\sigma_i$ below differ from
  $\sigma_1$ in at most four consecutive places, the assumption that
  the nullity of $M$ is at least three implies that an $r$-interval in
  $\sigma_i$ cannot differ from its counterpart in $\sigma_1$ at both
  ends.

  Consider $\sigma_3=(1,\underline{3},\underline{4},\underline{2},
  5,\ldots,n)$.  The relevant intervals are
  \begin{itemize}
  \item[$\diamond$] $\{4,2,5,6,\ldots,r+2\}$ (compare to $H_1$),
  \item[$\diamond$] $\{n-r+4,\ldots,n,1,3,4\}$ (compare to $H_2$), and
  \end{itemize}
  \begin{equation*}
    H_3=\{n-r+3,\ldots,n,1,3\}.
  \end{equation*}
  Thus, $\sigma_3$ has the desired properties unless $H_3$ is a
  circuit-hyperplane, so we assume it is.

  Consider $\sigma_4=(1,\underline{4},3,\underline{2}, 5,\ldots,n)$.
  The relevant intervals are
  \begin{itemize}
  \item[$\diamond$] $\{n-r+4,\ldots,n,1,4,3\}$ and
    $\{n-r+3,\ldots,n,1,4\}$ (compare to $H_3$), and
  \end{itemize}
  \begin{equation*}
    H_4=\{3,2,5,6,\ldots,r+2\}.
  \end{equation*}
  Thus, $\sigma_4$ has the desired properties unless $H_4$ is a
  circuit-hyperplane, so we assume it is.

  Consider $\sigma_5=(\underline{3},\underline{4},\underline{1},
  \underline{2},5,\ldots,n)$.  The relevant intervals are
  \begin{itemize}
  \item[$\diamond$] $\{1,2,5,6,\ldots,r+2\}$ (compare to $H_4$),
  \item[$\diamond$] $\{n-r+4,\ldots,n,3,4,1\}$ (compare to $H_2$),
  \item[$\diamond$] $\{n-r+3,\ldots,n,3,4\}$ and
    $\{n-r+2,\ldots,n,3\}$ (compare to $H_3$), and
  \end{itemize}
  \begin{equation*}
    H_5=\{4,1,2,5,6,\ldots,r+1\}.
  \end{equation*}
  Thus, $\sigma_5$ has the desired properties unless $H_5$ is a
  circuit-hyperplane, so we assume it is.

  Consider $\sigma_6=(\underline{4},\underline{3},\underline{1},
  \underline{2},5,\ldots,n)$.  The relevant intervals are
  \begin{itemize}
  \item[$\diamond$] $\{1,2,5,6,\ldots,r+2\}$ (compare to $H_4$),
  \item[$\diamond$] $\{3,1,2,5,6,\ldots,r+1\}$ (compare to $H_5$),
  \item[$\diamond$] $\{n-r+4,\ldots,n,4,3,1\}$ (compare to $H_2$),
  \item[$\diamond$] $\{n-r+3,\ldots,n,4,3\}$ (compare to $H_3$), and
  \end{itemize}
  \begin{equation*}
    H_6=\{n-r+2,\ldots,n,4\}.
  \end{equation*}
  Thus, $\sigma_6$ has the desired properties unless $H_6$ is a
  circuit-hyperplane, so we assume it is.

  Finally, consider $\sigma=(\underline{2},\underline{3},
  \underline{4}, \underline{1},5,\ldots,n)$.  The relevant intervals
  are
  \begin{itemize}
  \item[$\diamond$] $\{4,1,5,6,\ldots,r+2\}$ (compare to $H_1$),
  \item[$\diamond$] $\{3,4,1,5,6,\ldots,r+1\}$ (compare to $H_5$),
  \item[$\diamond$] $\{n-r+4,\ldots,n,2,3,4\}$ (compare to $H_2$),
  \item[$\diamond$] $\{n-r+3,\ldots,n,2,3\}$ (compare to $H_3$), and
  \item[$\diamond$] $\{n-r+2,\ldots,n,2\}$ (compare to $H_6$).
  \end{itemize}
  Thus, $\sigma$ has the desired properties, which completes the proof.
\end{proof}

We now turn to Conjecture~\ref{conj:cyclic}.

\begin{thm}
  Conjecture~\ref{conj:cyclic} holds for sparse paving matroids.
\end{thm}
  
\begin{proof}
  Consider disjoint bases $B=\{b_1,b_2,\ldots,b_r\}$ and
  $C=\{c_1,c_2,\ldots,c_r\}$ of a sparse paving matroid $M$.  By the
  basis-exchange property, we may assume that in the cycle
  $$\sigma=(b_1,b_2,\ldots,b_r,c_1,c_2,\ldots,c_r),$$ every $r$-interval
  of the form $\{b_i,b_{i+1},\ldots,b_r,c_1,\ldots,c_{i-1}\}$ is a
  basis; such cycles are said to \emph{start properly}.  We say that a
  \emph{problem occurs at $c_i$} if
  $\{c_i,c_{i+1},\ldots,c_r,b_1,\ldots,b_{i-1}\}$ is not a basis;
  clearly, $i>1$.  We will show how, if a problem occurs at $c_i$,
  then we can switch a few elements so that the number of problems
  decreases and the cycle starts properly; iterating this procedure
  produces the desired cycle.

  First assume $1<i<r$.  We will show that one of the following cycles
  starts properly and has fewer problems (we underline the few
  elements that are permuted):
  \begin{enumerate}
  \item[] $\sigma_1=(b_1,b_2,\ldots,b_r,c_1,c_2,\ldots,
    \underline{c_i}, \underline{c_{i-1}},\ldots,c_r)$,
  \item[] $\sigma_2=(b_1,b_2,\ldots,\underline{b_i},
    \underline{b_{i-1}}, \ldots,b_r,c_1,c_2, \ldots,c_r)$,
  \item[] $\sigma_3=(b_1,b_2,\ldots,b_r,c_1,c_2,\ldots,
    \underline{c_{i+1}}, \underline{c_{i-1}},
    \underline{c_i},\ldots,c_r)$.
  \end{enumerate}
  Since $S_0=\{c_i,c_{i+1},\ldots,c_r,b_1,\ldots,b_{i-1}\}$ is a
  circuit-hyperplane, Lemma~\ref{lem:chb} implies that
  $\{c_{i-1},c_{i+1},\ldots,c_r,b_1,\ldots,b_{i-1}\}$ is a basis.
  Only one other $r$-interval in $\sigma_1$ differs from its
  counterpart in $\sigma$, namely,
  $S_1=\{b_i,\ldots,b_r,c_1,\ldots,c_{i-2},c_i\}$, so it follows that
  $\sigma_1$ starts properly and has fewer problems than $\sigma$
  unless $S_1$ is a circuit-hyperplane.  Assume $S_1$ is a
  circuit-hyperplane.  Only two $r$-intervals in $\sigma_2$ differ
  from their counterparts in $\sigma$; of these, the set
  $\{c_i,c_{i+1},\ldots,c_r,b_1,\ldots,b_{i-2},b_i\}$ is a basis by
  Lemma~\ref{lem:chb} (compare it to $S_0$); if its complement,
  $S_2=\{b_{i-1},b_{i+1},\ldots,b_r,c_1,\ldots,c_{i-1}\}$, is a basis,
  then $\sigma_2$ starts properly and has fewer problems than
  $\sigma$, so we may assume that $S_2$ is also a circuit-hyperplane.
  Four $r$-intervals in $\sigma_3$ differ from their counterparts in
  $\sigma$, namely,
  $$T_1=\{c_{i-1},c_i,c_{i+2},\ldots,c_r,b_1,\ldots,b_{i-1}\}, \quad
  T_2=\{c_i,c_{i+2},\ldots,c_r,b_1,\ldots,b_i\},$$ and their
  complements.  Each of these sets is a basis by Lemma~\ref{lem:chb}
  since each symmetric difference $T_1\triangle S_0$, $T_2\triangle
  S_0$, $(E(M)-T_1)\triangle S_1$, and $(E(M)-T_2)\triangle S_2$ has
  two elements, so $\sigma_3$ starts properly and has fewer problems
  than $\sigma$.

  Now assume $i=r$, so $S_0=\{c_r,b_1,\ldots,b_{r-1}\}$ is a
  circuit-hyperplane.  Consider
  \begin{enumerate}
  \item[] $\sigma_1=(b_1,b_2,\ldots,b_r,c_1,c_2,\ldots,
    \underline{c_r}, \underline{c_{r-1}})$,
  \item[] $\sigma_2=(b_1,b_2,\ldots,\underline{b_r},
    \underline{b_{r-1}},c_1,c_2, \ldots,c_r)$,
  \item[] $\sigma_3=(b_1,b_2,\ldots,b_r,c_1,c_2,\ldots,
    \underline{c_{r-1}}, \underline{c_r}, \underline{c_{r-2}})$.
  \end{enumerate}
  An argument similar to that above shows that $\sigma_1$ starts
  properly and has fewer problems than $\sigma$ unless $S_1 =
  \{b_r,c_1,c_2,\ldots,c_{r-2},c_r\}$ is a circuit-hyperplane;
  likewise, $\sigma_2$ starts properly and has fewer problems than
  $\sigma$ unless $S_2 = \{{b_{r-1}},c_1,c_2, \ldots,c_{r-1}\}$ is a
  circuit-hyperplane.  Assume both $S_1$ and $S_2$ are
  circuit-hyperplanes.  Only four $r$-intervals in $\sigma_3$ differ
  from their counterparts in $\sigma$, namely:
  $$T_1=\{c_r,c_{r-2},b_1,\ldots,b_{r-2}\},\quad
  T_2=\{c_{r-2},b_1,\ldots,b_{r-1}\},$$ and their complements.  These
  sets are bases since $T_1\triangle S_0$, $T_2\triangle S_0$,
  $(E(M)-T_1)\triangle S_2$, and $(E(M)-T_2)\triangle S_1$ each have
  two elements, so $\sigma_3$ is the desired cycle on $B\cup C$.
\end{proof}

\section{Sparse Paving Matroids and the Number of Cyclic
  Flats}\label{sec:ncf}

A set $X$ in a matroid $M$ is \emph{cyclic} if $M|X$ has no coloops.
Such sets are precisely the (possibly empty) unions of circuits of
$M$.  Let $\mcZ(M)$ be the set of cyclic flats of $M$.  As noted
in~\cite{aff}, the cyclic flats, along with their ranks, determine the
matroid; indeed, this data can be seen as distilling the essential
geometric information about a matroid (see \cite{st,int} for
constructions that exploit this perspective).  Cyclic flats play many
roles in matroid theory, especially in the theory of transversal
matroids (see, e.g., \cite{trl,fund,aff,ing}). 

Let $z_n$ be $\max\{|\mcZ(M)|\,:\,|E(M)|=n\}$, that is, $z_n$ is the
greatest number of cyclic flats that any matroid on $n$ elements can
have.  In~\cite{cyc}, the problem of finding $z_n$ was raised.  The
importance of this problem stems from the fact that cyclic flats and
their ranks generally provide a relatively compact description of a
matroid.

To deduce a simple upper bound on $z_n$, let $a_i$ be the number of
$i$-element cyclic flats in a matroid $M$ with $|E(M)|=n$.  Note that
for $F\in\mcZ(M)$ and $e\in F$, the closure $\cl(F-e)$ is $F$; also,
for $x\in E(M)-F$, the restriction $M|(F\cup x)$ has a unique coloop,
namely $x$.  It follows that the sets $F$ and $F-e$, with $F\in
\mcZ(M)$ and $e\in F$, all differ, as do the sets $F$ and $F\cup x$
with $F\in\mcZ(M)$ and $x\in E(M)- F$.  Thus,
\begin{equation*}
  \sum_{i=0}^n a_i(i+1)\leq 2^n \qquad \text{ and } \qquad
  \sum_{i=0}^n a_i(n-i+1)\leq 2^n.
\end{equation*} 
Adding these inequalities gives $(n+2)\,|\mcZ(M)|\leq 2^{n+1}$, so
$z_n\leq 2^{n+1}/(n+2)$.  We next review a construction that yields
sparse paving matroids in which the number of cyclic flats is not so
far from this bound.

As mentioned in Section~\ref{sec:intro}, Knuth~\cite{knuth}
constructed a family of at least $\bigl(2^{\binom{n}{\lfloor n/2
    \rfloor}/2n}\bigr)/n!$ nonisomorphic sparse paving matroids of
rank $\lfloor n/2 \rfloor$ on $n$ elements.  To do this, he showed
that there is a sparse paving matroid of rank $\lfloor n/2 \rfloor$ on
$n$ elements with at least $\binom{n}{\lfloor n/2 \rfloor}/2n$
circuit-hyperplanes; the circuit-hyperplane relaxations of this
matroid, taking into account potential isomorphisms, give the family.

While exploring an equivalent problem in the context of coding theory,
Graham and Sloane~\cite{gs} generalized and strengthened Knuth's
result by showing that for each rank $r$ with $r\leq n$, there is a
sparse paving matroid of rank $r$ on $n$ elements with at least
$\binom{n}{r}/n$ circuit-hyperplanes.  Their construction, which we
sketch, has the same general flavor as Knuth's.  Partition the set of
all $0$-$1$ vectors $(a_0,a_1,\ldots,a_{n-1})$ of length $n$ with $r$
ones into $n$ classes according to the remainder, modulo $n$, of the
sum of the positions that contain ones, i.e., $\sum_i a_i\,i$.  They
noted that any two vectors in the same class differ in at least four
places.  At least one of the classes has at least $\binom{n}{r}/n$
vectors; by interpreting these vectors as the characteristic functions
of the circuit-hyperplanes, this class defines a sparse paving matroid
with at least $\binom{n}{r}/n$ circuit-hyperplanes.

The cyclic flats of a sparse paving matroid $M$ having rank and
nullity at least two are $\emptyset$, $E(M)$, and its
circuit-hyperplanes.  A routine induction (treating even $n$ and odd
$n$ separately) shows $\binom{n}{\lfloor n/2 \rfloor}\geq
2^{n-1}/\sqrt{n}$ (consistent with Stirling's approximation).  Thus,
it follows from Graham and Sloane's work that some sparse paving
matroid on $n$ elements has at least $2^{n-1}/n^{3/2} + 2$ cyclic
flats.  (For large $n$, the numbers of cyclic flats in these examples
far surpass those mentioned in~\cite{cyc}.)  We summarize these
remarks in the result below, which, if we apply $\log_2$ to each term
in the inequality, bears a strong resemblance to
inequality~(\ref{eq:bound}).

\begin{thm}\label{thm:zn}
  The maximum number of cyclic flats among matroids on $n$ elements,
  $z_n$, satisfies
  $$\frac{2^{n-1}}{n^{3/2}}+2\leq z_n \leq \frac{2^{n+1}}{n+2}.$$
\end{thm}

To close, we note that Graham and Sloane's examples cannot be
substantially improved upon within the class of sparse paving
matroids.  The sparse paving matroids that they construct have
$\binom{n}{\lfloor n/2 \rfloor}/n$ circuit-hyperplanes.  It is routine
to check that the right side of inequality (\ref{in:ch}) above is
maximized when $r = \lfloor n/2\rfloor$.  The ratio of this upper
bound to the number of circuit-hyperplanes in Graham and Sloane's
examples tends to $2$ as $n$ goes to infinity.  (Also see~\cite[Remark
2]{gs}.) This supports the natural suspicion that the lower bound in
Theorem \ref{thm:zn} is close to optimal.

\vspace{5pt}

\begin{center}
  \textsc{Acknowledgements}
\end{center}

\vspace{3pt}

I thank Anna de Mier and Jay Schweig for discussions about
Conjectures~\ref{conj:faber}--\ref{conj:cyclic} and for helpful
comments about this paper, and Gordon Royle for drawing my attention
to the results in~\cite{gs} and the possibility that sparse paving
matroids may maximize the number of cyclic flats.

\end{document}